\documentclass{article}
\usepackage{amsmath}
\usepackage{amsfonts}
 \usepackage{paralist}
 \usepackage{graphics}
 \usepackage{subfig}
 \usepackage{float}
 \usepackage[colorlinks=true]{hyperref}

\hypersetup{urlcolor=blue, citecolor=red}

\usepackage{epsfig}
\usepackage[final]{changes}

\newtheorem{theorem}{Theorem}[section]

\newtheorem{proposition}[theorem]{Proposition}
\newtheorem{corollary}[theorem]{Corollary}

\newtheorem{remark}[theorem]{Remark}

\newtheorem{proof}{Proof}

\title{Nonsmooth Frameworks for an Extended Budyko Model}

\author{Anna M. Barry$^*$, Esther Widiasih, and Richard McGehee}

\begin{document}
\maketitle

\centerline{\scshape Anna M. Barry$^*$}
\medskip
{\footnotesize
 \centerline{Department of Mathematics}
   \centerline{The University of Auckland}
   \centerline{Private Bag 92019}
   \centerline{ Auckland 1142, New Zealand}
} 

\medskip

\centerline{\scshape Esther Widiasih}
\medskip
{\footnotesize
 \centerline{Mathematics and Science Subdivision}
 \centerline{University of Hawaii, West Oahu}
\centerline{91-1001 Farrington Highway}
\centerline{Library Room 203}
\centerline{Kapolei, HI 96707 }
}

\medskip

\centerline{\scshape Richard McGehee}
\medskip
{\footnotesize
 \centerline{Department of Mathematics}
 \centerline{University of Minnesota}
\centerline{206 Church St SE}
\centerline{Minneapolis, MN 55455, USA}
}

\bigskip

\begin{abstract} 

In latitude-dependent energy balance models, ice-free and ice-covered conditions form physical boundaries of the system.  With carbon dioxide treated as a bifurcation parameter, the resulting bifurcation diagram is nonsmooth with curves of equilibria and boundaries forming corners at points of intersection.  Over long time scales, atmospheric carbon dioxide varies dynamically and the nonsmooth diagram becomes a set of quasi-equilibria.  {However, when introducing carbon dynamics, care must be taken with the physical boundaries and appropriate boundary motion specified.  In this article, we extend an energy balance model to include slowly varying carbon dioxide and develop a nonsmooth framework based on physically relevant boundary dynamics.  Within this framework, we prove existence and uniqueness of solutions, as well as invariance of the region of phase space bounded by ice-free and ice-covered states.} 
\end{abstract}

\section{Introduction}

There is significant evidence that large glaciation events took place during the Proterozoic era (2500-540 million years ago).  In particular, this evidence points to the existence of glacial formations at low latitudes, see the review articles \cite{hoffmanschrag2002terranova, pierrehumbert2011climate} and the references therein. One theory on the exodus from such an extreme climate was put forward by Joseph Kirschvink \cite{kirschvink1992protero}, who advocated that there was accumulation of greenhouse gases in the atmosphere, e.g. CO$_2$.  His theory purports that during a large glaciation chemical weathering processes would be shut down, thus eliminating a CO$_2$ sink.  Moreover, volcanic activity would continue during the glaciated state.  The combination of these effects would \textit{slowly} lead to enough build-up of atmospheric carbon dioxide to warm the planet and start the melting of the glaciers. Once a melt began, a deglaciation would follow \textit{rapidly} due to ice-albedo feedback.

There has been a wealth of modelling work on ``snowball'' events ranging from computationally intensive global circulation models (GCMs) to low dimensional conceptual climate models (CCMs).  In 1969, Mikhail Budyko and William Sellers independently proposed energy balance models (EBMs); these were CCMs capturing the evolution of the temperature profile of an idealized Earth \cite{budyko2010effect, sellersglobal}. Many others, for example \cite{ caldeira1992susceptibility, abbot2011, north1975theory}, have followed in the footsteps of Budyko and Sellers and used similar conceptual models capable of exhibiting snowball events. The low dimensionality of CCMs allows for a dynamical systems analysis, and hence a deeper investigation into some of the key feedbacks such as greenhouse gas and the ice-albedo effect. 

From the point of view of dynamical systems, many of these early works share a similar theme by focusing on a bifurcation analysis with respect to the \textit{radiative forcing} parameter, one that depends on changes in atmospheric CO$_2$ and other greenhouse gas levels.  The reader may find figures similar to those displayed in Figure \ref{icf} in earlier works \cite{abbot2011, hoffmanschrag2002terranova, pollard2005snowball}.  These figures illustrate the glaciation state of an Earth with symmetric ice caps in terms of atmospheric CO$_2$ i.e. ice line latitude at $90^o$ means Earth is ice free while $0^o$ means Earth is fully glaciated. In both figures, the effect of the radiative forcing due to CO$_2$ and other greenhouse gases is treated statically as a parameter in a simple bifurcation analysis.  Mathematically, the bifurcation analysis already extends beyond the realm of smooth dynamical systems. On one hand, the bifurcation diagrams are obtained conventionally, with the bold curves indicating the stable branches and the dotted curves showing the unstable ones.   On the other hand, the extreme states (ice-free and ice-covered) are not true equilibria of the system, but treated as such in the literature as evidenced by the labeling of ``ice-free branch'' and ``ice-covered branch''. Indeed, the extreme states of the ice line are special because they serve as physical boundaries for the dynamics: the ice line cannot extend beyond the pole nor, because of the north-south symmetry assumed in the models, can it move downward of the equator. Therefore, mathematical models of Snowball Earth must reflect this physical imposition and be treated with a nonsmooth systems perspective.

Since glacial extent varies over time and dynamic processes such as chemical weathering affect the level of atmospheric CO$_2$, the bifurcation diagrams in Figure \ref{icf}  can be viewed as phase planes with dynamic variables consisting of the glacier extent (ice line) and  radiative forcing due to greenhouse gases. In particular, if a global glaciation did occur, and Kirschvink's argument about the accumulation of atmospheric  CO$_2$ held, then we should expect orbits of this dynamical system to traverse the extreme ice latitudes, i.e. the equator and the pole.  Following Kirschvink's idea further, we treat greenhouse gas effects on the energy balance by incorporating a slow CO$_2$ variable and treat the ice latitude as a relatively fast variable. The goal of the present article is to analyze such an interplay using a nonsmooth systems framework, thus providing mathematical support for Kirschvink's hypothesis.  In this work, we treat the bifurcation diagram (similar to those in Figure \ref{icf}) as a set of quasi-steady states of a slow-fast system and {specify boundary motion based on Kirschvink's hypothesis and the long-term carbon cycle.}  

 \begin{figure}[ht]
  \centering
   \subfloat[Figure 10 from \cite{abbot2011} ]  {\includegraphics[width=0.45\textwidth]{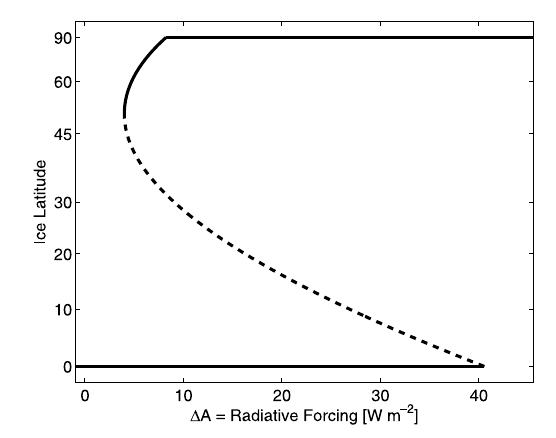} {\label{fig11}} }  \quad \quad
   \subfloat [Figure 6 from \cite{hoffmanschrag2002terranova}]  {\includegraphics[width=0.45\textwidth]{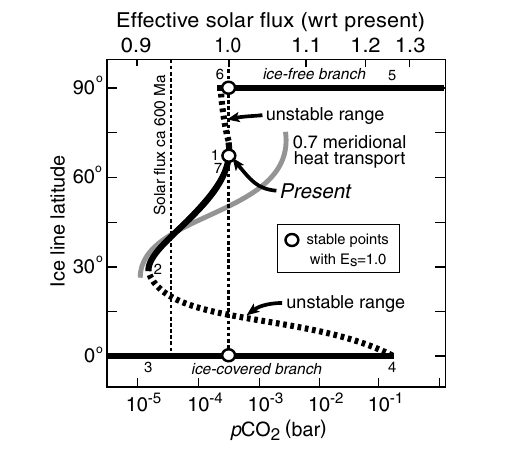} {\label{fig12}}} \quad
  \caption{Bifurcation diagrams from energy balance models illustrating hysteresis in the climate system.  In each figure, solid lines correspond to stable steady states while dashed lines correspond to unstable steady states.  The positive horizontal axis can be thought of as increasing atmospheric carbon dioxide, and the vertical axis is the latitude of the ice line.  Stability of snowball and ice-free states is \textit{inferred}; these are physical boundaries and not true equilibria of the equations.}  
\label{icf}
 \end{figure}

We present the topics as follows.  In the next section, we motivate the model of interest.  In Section \ref{sect:non}, we propose an extension of the model that is amenable to a nonsmooth dynamical systems treatment.  In particular, we show that a nonsmooth ice line model based on a latitude-dependent EBM coupled with a simple equation for greenhouse gas evolution has unique forward-time solutions.  We also show that a similar result can be obtained by embedding the system in the plane and utilizing Filippov's theory for differential equations with discontinuous right-hand sides in Section \ref{sect:alt}. We end Section \ref{sect:non} with an analysis of the system dynamics and conclude the paper with a discussion in Section \ref{sect:discussion}.

\section{The Equations of Motion}
\subsection{The Budyko Energy Balance Model and the Ice Line Equation}
The Budyko EBM describes the evolution of the annual temperature profile $T=T(y,t)$, where $t$ denotes time and $y$ denotes the sine of the latitude. The governing equation may be written:

\begin{equation} \label{bud}
R \frac{\partial }{\partial t}T(y,t) =  Qs(y)(1-\alpha(\eta,y))-(A+BT(y,t))+C \left( \int_0^1 T(y,t)dy-T(y,t) \right).
\end{equation}

The main idea is that the change in temperature is proportional to the imbalance in the energy received by the planet.  The amount of short wave radiation entering the atmosphere is given by $Qs(y)(1-\alpha(\eta,y))$, where Q is the total solar radiation (treated as constant), $s(y)$ is the distribution of the solar radiation, and $\alpha(\eta,y)$ is the albedo at latitude $y$ given that the ice line is at latitude $\eta$. The outgoing longwave radiation is the term $A+BT(y,t)$; it turns out that the highly complex nature of greenhouse gas effects on the Earth's atmosphere can be better approximated by a linear function of surface temperature than by the Stephan-Boltzmann law for blackbody radiation ($\sigma T^4$), see the discussion in \cite{graves1993new}. The parameter $A$ is particularly interesting here, because it is related to greenhouse gas effects on the climate system, which we will describe further in Section \ref{sec-ghg}.  The transport term $C ( \int_0^1 T(y,t)dy - T )$ redistributes heat by a relaxation process to the global average temperature. {Note that Budyko's equation assumes symmetry of the hemispheres and so one typically works only with the northern hemisphere, i.e. $y\in [0,1]$.}  A more detailed discussion of this model can be found in \cite{tung2007topics}. Table \ref{tab:Fun} lists the polynomial expressions of some key functions in this model using standard parameter values.  

The evolution of the ice-water boundary, or what is often called the \textit{ice line}, $\eta$, provides a positive feedback to the system through the albedo function $\alpha(\eta,y)$.  In \cite{tung2007topics} and \cite{north1975theory}, a critical ice line annual average temperature, $T_c$, is specified.  Above this temperature ice melts, causing the ice line to retreat toward the pole, and below it ice forms, allowing glaciers to advance toward the equator. One way to model this is described in \cite{widiasih2013dynamics}, where the augmented ice line equation governing $\eta$ is written as

\begin{equation} \label{etadot}
\frac{d \eta}{dt} = \rho ( T(\eta)-T_c ).
\end{equation}

{In \cite{mcgwid2014simplification}, McGehee and Widiasih imposed equatorial symmetry on the set of solutions to \eqref{bud}-\eqref{etadot} by expanding the temperature profile $T(y,t)$ in even Legendre polynomials in the spatial variable $y$.  They found that the space of even Legendre polynomials contains all equilibrium solutions of the model and is invariant for the system when using the quadratic approximation to the insolation distribution $s(y)$ that was proposed by North \cite{north1975theory}. By using only the first two terms in the expansion, the authors reduced the infinite-dimensional system \eqref{bud}-\eqref{etadot} to one consisting of five ordinary differential equations, four representing the temperature profile and one for the ice line.} {They also found that the five-dimensional temperature profile-ice line system could be further reduced to a single equation, due to the existence of a one-dimensional invariant manifold parametrized by the ice line. The equation they obtained is}

\begin{equation} \label{etadot-reduced}
\frac{d\eta}{dt}= h_0(\eta; A)
\end{equation}
\noindent with
\begin{equation}\label{h} 
h_0(\eta;A):= \rho\left(\frac{Q}{B+C} \left(s(\eta)(1-\alpha(\eta,\eta))+\frac{C}{B}(1-\overline{\alpha}(\eta))\right)-\frac{A}{B}-T_c \right)
\end{equation} 
\noindent where $\overline{\alpha}(\eta)=\int_0^1 s(y) \alpha(\eta,y)dy$. The parameter $A$ appearing in \eqref{etadot-reduced} is indeed the greenhouse gas parameter in \cite{mcgwid2014simplification}, and is playing a similar role as that in the bifurcation diagram shown in Figure \ref{fig11}.

{In the same work, they examined the consistency of the coupled temperature profile-ice line system \eqref{bud}-\eqref{etadot} with physical interpretation. The orbital (i.e. Milankovitch) forcings affect the amount of sunlight that the planet receives, hence, its temperature. However, the planet's response to these forcings is delayed. The Budyko equation has been shown to be a good approximation to Earth's temperature with the assumption of a 2500 year delay between the orbital forcings and the climate response \cite{mcglehman2012}. Building on this assumption, the authors in \cite{mcgwid2014simplification} show that the parameter $\rho$ must be small, and approximate it to be about 0.001.}

In what follows, we {couple an equation for atmospheric greenhouse gas evolution through the (former) parameter $A$.  We couple this equation to the reduced model of McGehee and Widiasih \cite{mcgwid2014simplification} and use this as a proof of concept for the nonsmooth questions of interest. Numerical simulations produce similar dynamics in higher-dimensional versions of the model, but making a rigorous reduction  of the infinite-dimensional system to the nonsmooth system studied in this article would require new center manifold results for discontinuous vector fields.  We leave this for future work.} From here on, we highlight the explicit dependence of the equation for $\eta$ on $A$ by writing 
\begin{equation*}
\frac{d\eta}{dt}=h(A,\eta)
\end{equation*}
where $h(A,\eta)=h_0(\eta;A)$ and $h(A,\eta)$ is thought of as a real valued function over the plane $\mathbb{R} \times \mathbb{R}$.

\begin{table}[htb]
\centering
\begin{tabular}{|ccc|}  
\hline
\textbf{Parameters} & Value & Units  \\ \hline \hline
&& \\
$Q$  &  321  &  $\text{W}\text{m}^{-2}$ \\
$s_1$ & 1 & dimensionless  \\
$s_2$ &  -0.482  &  dimensionless  \\
$B$ & 1.5  &  $\text{W}\text{m}^{-2}\text{K}^{-1}$    \\
$C$ & 2.5B & $\text{W}\text{m}^{-2}\text{K}^{-1}$  \\
$\alpha_1$ & 0.32 & dimensionless   \\
$\alpha_2$ & 0.62 & dimensionless \\
$T_c$ & $-10$ & ${}^\circ\text{C}$   \\
&& \\
\hline
\end{tabular}  
\caption{Parameter values as in \cite{abbot2011}}   \label{tab:ParValues}
\begin{tabular}{|c|}
\hline
  \textbf{Functions} \\ \hline \hline
   \\
$s(y) = 1 - \frac{0.482}{2} (3 y^2 - 1)$\\  
$h(A,\eta)=\rho\left(112.88+56.91\eta-24.31\eta^2-11.05\eta^3-\frac{A}{1.5}\right)$   \\
 $g(A,\eta)=\delta(\eta-\eta_c)$  \\
$\alpha(\eta,y)=\begin{cases} 
&\alpha_1 \text { when } y< \eta \\
& \frac{\alpha_1+\alpha_2}{2} \text{ when } y=\eta\\
&\alpha_2 \text{ when } y>\eta 
  \end{cases}$ \\
 \\
\hline
\end{tabular}  
\caption{Functions as in \cite{mcgwid2014simplification}.}   \label{tab:Fun}
\end{table}

\subsection{Incorporating Greenhouse Gases} \label{sec-ghg} 

We now make an argument for a very simple form of an equation for greenhouse gas evolution. In the Budyko and Sellers models \cite{  budyko2010effect,sellersglobal}, the parameter $A$ plays the important role of reradiation constant.  The Earth absorbs shortwave radiation from the sun, and some of this is reradiated in the form of longwave radiation.  The current value of $A$ is measured using satellite data to be approximately $202$ $W/m^2$ \cite{ graves1993new, tung2007topics}. 

However, throughout the span of millions of years the longwave radiation parameter $A$ is not constant, as the amount of heat reradiated to space depends crucially on greenhouse gases, especially carbon dioxide. {This relationship is highly complex, but there have been some attempts at creating simplified models.  One example comes from \cite{caldeira1992susceptibility}, where $A$ was expressed as a polynomial expansion in the logarithm of atmospheric carbon dioxide for CO$_2$ near modern values measured in parts per million. Such an expansion is unlikely to directly apply over the time period associated with entrance into and exit from Snowball Earth, where carbon dioxide would have varied over many orders of magnitude. Following \cite{abbot2011} (see also Figure \ref{fig11}), we choose to vary $A$ as a proxy for CO$_2$ rather than attempting to build a model for their direct relationship.} Intuitively, adding carbon dioxide to the atmosphere decreases its emissivity, allowing less energy to escape into space. Therefore, outgoing longwave radiation should vary inversely with CO$_2$.

Since the land masses were concentrated in middle and low latitudes prior to the global glaciation period, Kirschvink postulated that the stage was set for an ice-covered Earth.  Then \lq On a snowball Earth, volcanoes would continue to pump CO$_2$ into the atmosphere (and ocean), but the sinks for CO2 – silicate weathering and photosynthesis would be largely eliminated\rq   \cite{ hoffmanschrag2002terranova, kirschvink1992protero}. In  more recent work,  Hogg \cite{hogg} put forth an elementary model  for the evolution of greenhouse gases consistent with Kirschvink's theory.  In short, he argued that the main sources for atmospheric CO$_2$ were due to an averaged volcanism rate and ocean outgassing.  The main carbon dioxide sink was said to be due to the weathering of silicate rocks. For our purposes, it will be enough to work with volcanism and weathering.  Let $V$ denote the rate of volcanism and $W$ the weathering rate, where the latter is assumed to depend on the location of the iceline, or more specifically the amount of available land or rock to be weathered.  Putting this together with the assumption that the reradiation variable $A$ varies inversely with CO$_2$ gives

\begin{align} 
\frac{dA}{dt}&=-(V-W\eta)\nonumber\\
&:=\delta(\eta-\eta_c). \label{Adoteq}
\end{align} 

where $\delta>0$ and $\eta_c=\frac{V}{W}$, the ratio of volcanism to weathering. Let $g(A,\eta):=\delta (\eta -\eta_c)$. 

Next, allowing $A$ to vary in the ice line equation \eqref{h}, we now have a system of equations for $A$ and $\eta$:
\begin{equation} \label{Aeta}
\begin{cases}
&\dot{A}=g(A,\eta)\\
&\dot{\eta}=h(A, \eta).\\
\end{cases}
\end{equation}

To understand the timescale of $A$, we refer to the elucidation of the snowball scenario by Hoffman and Schrag in which they argue that weathering took place at a much slower rate than the ice-albedo feedback (see Fig. 7 on page 137 \cite{hoffmanschrag2002terranova} ). 

The idea of packaging the essence of the long term carbon cycle into one simple equation is not novel, and is consistent with earlier findings \cite{  edmond1996fluvial, hogg, kump2000chemical}.  The novelty comes from connecting such an equation to an energy balance model and analyzing the dynamics of the coupled system.  Indeed, the parameter $A_{ex}$ on page 644 \cite{kump2000chemical} is the variable $\eta$ in equation \eqref{Aeta} and it represents \textit{the effective area of exposure of fresh minerals}. The parameter $\eta_c$ can therefore be thought of as a \textit{critical} area. The coupling of the ice line $\eta$ and the greenhouse gas variable $A$ follows naturally.  

In its current form, the model does not restrict dynamics to the physical region $0\leq \eta\leq 1$.  There are orbits that exit this interval on either end, and we must therefore create reasonable assumptions for projection of this motion onto the boundaries.  Furthermore, $\eta$ should evolve along the physical boundaries, as suggested by  the bifurcation diagrams in Figure \ref{icf}. {In the next section, we develop a nonsmooth version of the model that respects these boundaries and prove existence and uniqueness of solutions. We use this as a proof of concept for how to treat physical boundaries beyond which a vector field is undefined.} 

\section{ A {Nonsmooth} System for a Glaciated Planet\label{sect:non}}
In this section, we introduce suitable constraints on the ice dynamics so that trajectories with initial conditions in the physical region remain in this region for all time. {Moreover, we obtain physically reasonable motion on the boundaries.}

As mentioned previously, the pole and the equator define physical constraints of the ice line;  it must stay in the unit interval $[0,1]$. However, the model in its current form does not take this into account.  A natural choice is to set $\dot{\eta}=0$ when $\eta$ reaches a physical boundary.  In addition to this, one needs to account for the stability/instability of these states, thus allowing orbits to exit the boundary when it becomes unstable.  More specifically, \added{it would be natural to define the evolution of $(A,\eta)$ via}

\begin{equation}\label{AetaPW}
\begin{cases}\dot{A}&=g(A,\eta)\\
\dot{\eta}&= f(A,\eta)
\end{cases}
\end{equation}

where 

\begin{equation}\label{etaPW}
f(A,\eta)=\begin{cases}
 0 & \{ \eta=0 \text{ and } h(A,\eta)<0 \}\text{ \text{or} }\{ \eta=1 \text{ and } h(A, \eta)>0\}\\
   h(A, \eta) & \text{otherwise}\end{cases}
   \end{equation}
   
\added{and $g$, $h$ are as in Table \ref{tab:Fun}.}  
The second equation forces $\eta$ to stop at the physical boundary exactly when it is about to cross that boundary.  When this happens, $\eta$ should remain constant while $A$ continues to evolve.  As $A$ evolves, so does the value of $h$ and \added{hence $\eta$ should be allowed to re-enter the interval $(0,1)$} as soon as $h$ becomes positive on the lower boundary, $\eta=0$, or negative on the upper boundary, $\eta=1$. This is analogy with Kirschvink's hypothesis, as we expect orbits that enter the snowball state to recover when enough carbon dioxide has built up in the atmosphere. 

\added{However, the right-hand side of \eqref{AetaPW} is discontinuous at the boundaries and so solutions will not necessarily be differentiable there.  Thus, we \textit{define} a solution to an initial value problem for \eqref{AetaPW} with initial condition $(A(0),\eta(0))=(A_0,\eta_0)$ to be an absolutely continuous function $X(t)=(A(t),\eta(t))$ that satisfies the integral equation}

\begin{equation}\label{defn:sol}
X(t)=X_0+\int_0^t \left(g(A(s),\eta(s)), f(A(s),\eta(s))\right) ds
\end{equation}

with $f$ and $g$ as above.  \added{Note that since $h$ is smooth, the integral equation is equivalent to the system $(\dot{A},\dot{\eta})=(g,h)$ inside the strip $0<\eta<1$.  }

 \begin{remark}\label{rmk:projrule} {The discontinuous vector field arising from \eqref{AetaPW} defines a \textup{Projection Rule} for orbits that reach the boundary:  when the vector field points \textit{out} of the strip $\mathbb{R}\times [0,1]$, the forward evolution is defined by projection of the vector field onto the boundary, and when the vector field points \textit{into} the strip, the system is left to evolve without interruption.}
\end{remark}

{ Because there is no physically reasonable way to define the vector field outside of the interval $0\leq \eta\leq 1$, a system with such a rule does not fit readily into an existing analytical framework. Indeed, common frameworks for nonsmooth systems assume that the vector field is defined in a neighborhood of the curve along which it is discontinuous. Thus it is not clear that classical results from dynamical systems apply in this setting.  Fundamentally, the existence and uniqueness of solutions needs to be verified.  In what follows, we prove that the system with the rule specified above has unique forward-time solutions and that the physical region is forward invariant. We also outline an alternative approach that involves defining a nonphysical ``confinement'' vector field outside the physical region and serves to handle and correct the behavior of numerically generated solutions that may overshoot the boundary.  This second approach also guarantees existence, uniqueness, and forward invariance. Further, its solutions trace out the same curves in phase space as the previous approach, though they differ in their smoothness.}

\subsection{First framework: the projection rule\label{sect:ex-un}}

{In this section, we show that the projection rule results in existence and uniqueness of solutions.}

\begin{theorem} \label{thm:projection}
{\added{Let $0<\eta_c<1$ and $t>0$.} Then an initial value problem for \eqref{AetaPW} with initial condition $(A_0,\eta_0)$ and $0\leq  \eta_0\leq 1$ has a unique forward-time solution \added{of the form \eqref{defn:sol}}. \added{The time derivative of the solution has at most a finite number of discontinuities and satisfies \eqref{AetaPW} in between them.}  Furthermore, the strip $\mathbb{R} \times [0,1]$ is forward invariant. }
\end{theorem}

\begin{proof}
{ \added{First, a solution is created in an ad hoc manner by concatenating pieces of solutions to the associated smooth systems on each subregion of the phase space. Since  $f$ and $g$ are smooth on the open set $\{0< \eta<1\}$, we focus only on cases where entrance into or exit from $\eta=0,1$ occurs.} Without loss of generality, consider an initial condition $(A_0,\eta_0)$ with $0<\eta_0<1$ such that the smooth solution $(A(t),\eta(t))$ reaches the line $\eta=1$ at some time $t_1< t$, i.e. $\lim_{t\rightarrow t_1^-}\eta(t)=1$. }
  
  \added{ From Table \ref{tab:Fun}, we see that $h(A,\eta)=\frac{\rho}{1.5}(f(\eta)-A)$ where $f$ is a cubic function of $\eta$.  Thus, as can be seen in Figure \ref{critman}, the set $\{h=0\}$ divides the strip into two regions; $h>0$ on the left-most region and $h<0$ on the right-most region.} Define $A_1=\lim_{t\rightarrow t_1^-}A(t)$ so that $h(A_1,1)\geq 0$.  If $h(A_1,1)>0$, then $h$ being Lipschitz continuous implies that it must be positive in a neighborhood of $(A_1,1)$. Therefore, we extend the solution according to \eqref{etaPW} and \eqref{defn:sol} via
  
   \begin{align}
  A(s)&=A_1+\delta(1-\eta_c)(s-t_1)\label{defn:bdry1}\\
 \eta(s)&=1\label{defn:bdry2}
  \end{align}
    
\added{for $t_1\leq s\leq t_2$ where $A_2=\lim_{t\rightarrow t_2^-} A(t)$ is such that $h(A_2,1)=0$, i.e. $t_2=t_1+\frac{f(1)-A_1}{\delta(1-\eta_c)}$. }
  
 Since $0<\eta_c<1$, note that $\delta(1-\eta_c)>0$ and hence there is a single orbit satisfying the smooth system of differential equations

\begin{equation}\label{eq:smooth}
\begin{cases}
\dot{A}&=g(A,\eta)\\
\dot{\eta}&=h(A,\eta)
\end{cases}
\end{equation}

\noindent {that passes the curve $h=0$ from left to right and touches the boundary at the single point $(A_2,1)$ forming a tangency.  Moreover, this orbit is strictly contained in the interval $0<\eta<1$ on either side of the tangency \added{and also satisfies the integral equation \eqref{defn:sol} where $f$ is replaced by $h$.} Thus, taking an initial condition $(A(t_2),\eta(t_2))=(A_2,1)$ and concatenating the forward evolution of this trajectory with the previous piece of the solution \added{to the nonsmooth system} provides a further extension via}

\begin{equation*}
\begin{cases}
A(s)&=A_2+\int_{t_2}^s g(A(\tau),\eta(\tau))\, d\tau\\
\eta(s)&=1+\int_{t_2}^s h(A(\tau),\eta(\tau))\, d\tau,
\end{cases}
\end{equation*}

{\noindent \added{for all} $t_2\leq s\leq t$ such that $0<\eta(s)<1$.  \added{Suppose $t_3$ is the next time the created solution reaches the boundary. In the case that the orbit again reaches the ice-free state $\eta=1$, we extend the solution exactly as above and this definition leads to another exit from this boundary at the point $(A_2,1)$ as before. In this case we have created a closed orbit. Analogous extensions can be made when the ice-covered boundary $\eta=0$ is reached. Any solution created by this method of concatenation that reaches the same boundary twice lies on a closed orbit. This in turn implies that there can only be a finite number of switches, i.e. entrances into and exits from the boundaries.} The created solution is absolutely continuous and differentiable everywhere except at the points where it enters one of the lines $\eta=0,1$. \added{By definition, it satisfies \eqref{AetaPW} everywhere except at these points.}}

{Now suppose there is another absolutely continuous solution $(\bar{A}(t),\bar{\eta}(t))$ of \eqref{AetaPW} with $(\bar{A}(0),\bar{\eta}(0))=(A_0,\eta_0)$. \added{Due to smoothness of $h$ and $g$}, both trajectories must reach the boundary at the same time and hence they agree for all $0\leq s<t_1$.  If they reach the boundary at the tangency point, then \added{the integrands of \eqref{defn:sol} continue to be $h$ and $g$ and hence classical uniqueness arguments can be used until such time that the solutions reach another boundary.}  If they reach the boundary at $(A_1,1)$ such that $h(A_1,1)>0$, then \added{both satisfy \eqref{defn:bdry1}-\eqref{defn:bdry2}} and uniqueness persists until $t_2$ such that $h(A(t_2),1)=0$. 

Indeed, since $g$ and $h$ are Lipschitz continuous, classical arguments invoking Gr{\"o}nwall's Inequality can be used to complete the uniqueness proof.  In particular, for $t_2\leq s\leq t$ such that $0<\eta(s)\leq 1$ one finds that}
\begin{equation*}
\lVert (A(t),\eta(t))-(\bar{A}(t),\bar{\eta}(t))\rVert^2\leq (K_1^2+K_2^2)\left(\int_{t_2}^t \lVert(A(s),\eta(s))-(\bar{A}(s),\bar{\eta}(s))\rVert ds\right)^2
\end{equation*}
{where $K_1$ and $K_2$ are the Lipschitz constants of $g$ and $h$.  \added{As discussed above, the solution may reach one or both boundaries a finite number of additional times, but the same uniqueness arguments apply for each segment. In particular, the only requirement is that the solutions agree at the beginning of each segment, and this is immediate the above arguments.} Thus $(A(t),\eta(t))=(\bar{A}(t),\bar{\eta}(t))$. Forward invariance is also immediate from the rules specified in \eqref{etaPW}.}

\end{proof}

{Note that Theorem \ref{thm:projection} only guarantees forward-uniqueness of solutions of \eqref{AetaPW}, and in general solutions are not unique in backward time; any solution that passes through a point on an attracting portion of the boundary (i.e. where $h(A,\eta)>0$ on $\eta=1$ or where $h(A,\eta)<0$ on $\eta=0$) could have have reached that point directly from the interval $0<\eta<1$ or from further back along the boundary itself.}

\begin{remark}The solutions guaranteed by Theorem \ref{thm:projection} are of the same type as those guaranteed by Carath\'eodory's Existence Theorem \cite{caratheodory2013vorlesungen} (see also \cite{filippov1988}). However, the existence theorem does not apply here because it requires that the right-hand sides of the differential equations be continuous in $(A,\eta)$.  In the above theorem, we have shown that the projection rule specified in \eqref{AetaPW} and Remark \ref{rmk:projrule} has enough continuity to guarantee existence and uniqueness of solutions. { More specifically, note that this rule guarantees upper semi-continuity of the vector field, which is also a key requirement for existence and uniqueness of Filippov's notion of solution \cite{filippov1988}, and which we explore in the next section. We will see that the solution given by the theorem above agrees with Filippov's solution, but the latter may have multi-valued derivative. }\end{remark}

\subsection{Alternative approach: Filippov framework \label{sect:alt}}
{In this section, we introduce a Filippov systems perspective by way of defining a \emph{confinement vector field} that points into the unit interval $0<\eta<1$ from outside. This perspective is useful for numerical simulation because it handles the case when orbits overshoot the physical boundary.   We wish to choose a confinement field that preserves the speed of motion along the boundary guaranteed by the projection rule \eqref{AetaPW} and obtain solutions which have the same forward evolution.   We will see that Filippov's convention \cite{filippov1988} preserves this speed if we work with the following piecewise-Lipschitz vector field on $\mathbb{R}\times\mathbb{R}$:}

\begin{align}\label{ig-pws}
&\dot{A} = g(A,\eta)\\
&\dot{\eta} = \begin{cases} 
-|h| \quad & \eta > 1\\
h \quad & 0< \eta < 1\\\label{ig-pws1}
|h| \quad &\eta < 0. \\
\end{cases}\\\nonumber
\end{align}
having an initial value $(A(0),\eta(0)) \in \mathbb{R} \times \mathbb{R}$.

{In the above scenario, the boundaries $\eta=0,1$ become curves along which the vector field is discontinuous (so-called \textit{switching manifolds}) and reasonable motion along these curves must be specified.  Because we have chosen to work with Filippov's convention, we take all convex combinations of the limiting vectors from either side, i.e. linear combinations of the Lipschitz vector fields. A more general convention allowing for nonlinear combinations was discussed in \cite{jeffrey2014hidden}, and taking such a perspective may require a different choice of external vector field to preserve the speed of boundary motion. We do not discuss the nonlinear case further here.  Instead, we consider the differential inclusion}

\begin{align}\label{diffinc}
\dot{\eta}\in \begin{cases}
 \{(1-q)h+q(-|h|),\; q\in[0.1]\} \quad & \eta=1\\
 \{(1-q)h+q|h|,	\; q\in[0,1]\} \quad & \eta=0
 \end{cases}
\end{align}
and let $f(A,\eta)$ be the vector field defined by right-hand sides of \eqref{ig-pws}-\eqref{diffinc}. A \textit{Filippov solution} to \eqref{ig-pws}-\eqref{diffinc} will be defined as an absolutely continuous function $X(t)=(A(t),\eta(t))$ such that $\dot{X}\in f(X)$ for almost all time $t$.

{We observe that the vector field defined by \eqref{ig-pws}-\eqref{ig-pws1} points ``into'' the interval $0<\eta<1$ except along the curve $h(A,\eta)=0$ where it is parallel to $\eta=0, 1$. Thus, according to Filippov's theory, there can only be \textit{sliding} or \textit{crossing} behavior along the boundary away from the two transverse intersections with $h=0$. All crossing orbits must enter the interval $0<\eta<1$. As we will see, the points where $h=0$ signal the transition from (attracting) sliding to crossing.
We now show the existence and uniqueness of Filippov solutions and the forward invariance of the strip. }

\begin{proposition} \label{filippov-bud} 
{There exists a unique Filippov solution to \eqref{ig-pws}-\eqref{diffinc}. Furthermore,  the strip $\mathbb R \times [0,1]$ is forward-time invariant. }
\end{proposition}

\begin{proof}

{This is a direct application of Filippov's existence and uniqueness result, Theorem 14, p. 228, in \cite{filippov1964}. Using similar notation, let $f(z)$ denote the vector-valued function $[g(z),h(z)]$, the normal direction $N$ is the vector $[0,1]$, and anything with a subscript $N$ denotes the projection onto $N$. By design, the confinement vector fields above and below the strip $\mathbb{R} \times [0,1]$ point into the strip. We consider a small neighborhood $G$ around a point $z_0$ away from the nullcline $h=0$ on the line $\eta=1$.  Letting $f^+$ and $f^-$ denote the vector field above and below $\eta=1$, we see that $f^+ - f^-$ is continuously differentiable, and the condition that either $f_N^+ < 0$ or $ f_N^- >0$ is satisfied for all points in $G$. Then Filippov's theorem asserts the existence and uniqueness of a forward-time solution to system \eqref{ig-pws}-\eqref{diffinc} as well as continuous dependence on initial conditions. Lastly, Remark 1, p. 229 \cite{filippov1964} asserts that the result holds at the intersection of $\eta=1$ with the $h$-nullcline, since at that point, $f^+ = 0 = f^-$.  A similar argument applies for initial conditions near $\eta=0$. }

{Lastly, since any solution can only cross into or slide along the boundary, the strip is forward invariant.}

\end{proof}

\begin{remark}\label{remark}
There are many possible ways of extending system \eqref{Aeta} beyond the strip $\mathbb{R} \times [0,1]$. The choice of extension as done in system \eqref{ig-pws}-\eqref{diffinc} assures the following points: 
\begin{enumerate}
\item The strip $\mathbb{R} \times [0,1]$ is attracting.
\item The switching from sliding to crossing is entirely determined by zeroes of $h(A, \eta)$. This point is especially important in the modeling of the physical system, since $h$ signals the ice line's advance or retreat and $h$ should  be ``off" for some positive amount of time at the extreme ice line locations (pole or equator) and should ``turn back on" when the system has sufficiently reduced or increased greenhouse gases.
\item {Filippov's theory asserts that at the attracting sliding segment on $\eta=1$, the time derivative must be $$ \frac{d\eta}{dt} = \frac{-|h|+ h}{2} $$ almost everywhere, while at $\eta=0$, it is $$\frac{d\eta}{dt} = \frac{|h|+ h}{2}$$ almost everywhere. Hence, the effective speed at which orbits slide along the boundaries is zero and the motion is thus entirely determined by the governing function of the greenhouse gas effect,  $g(A,\eta)$. Again, this aspect is especially important in the modeling of the physical system, since at the extreme ice line location (pole or equator), the ice-albedo feedback shuts off and the greenhouse gas effect is the only driver of the system. }
\end{enumerate}
\end{remark}

\subsection{Dynamics of the model \label{sect:dynamics}}
 {We now explore the dynamics of the system \eqref{AetaPW}.  We show that its boundary dynamics guarantee exit from the snowball scenario, and use the Filippov {perspective} to {run simulations.}  We find that the system possesses a stable small ice cap regime as well as large amplitude periodic orbits that oscillate between ice-covered and ice-free states.} 

The $\eta$-nullcline, given by $h(A,\eta)=0$, is cubic in $\eta$ and linear in $A$ (see Table \ref{tab:Fun}).  It has a single fold in the region $0 \leq \eta \leq 1$ at $\eta =\eta_f \approx 0.77$, as can be seen in Figure \ref{critman}. The $A$-nullcline is the horizontal line $\eta=\eta_c$. If $0<\eta_c<1$, the system has a single fixed point. The eigenvalues associated with this fixed point are

\begin{equation*}
\lambda_{\pm}=\frac{1}{2}\left(\frac{\partial h}{\partial \eta}\pm \sqrt{\left(\frac{\partial h}{\partial \eta}\right)^2-\frac{4\delta}{B}}\right).
\end{equation*}

 From this and Figure \ref{critman}, we see that if the fixed point lies above the fold and in the physical region, i.e. $\eta_f<\eta_c<1$, then it is stable, and it is unstable when $0<\eta_c<\eta_f$.  This is reminiscent of the ``slope-stability theorem'' from the climate literature \cite{cahalan1979stability}. In that work, the authors related changes in slopes of equilibrium curves to changes in stability of equilibria for a globally averaged energy balance model.  They found that a small ice cap or ice-free solution could be stable, a large ice cap was unstable, and a snowball state was again stable. 

\begin{remark} In the case that $\eta_c=0,1$, one of the physical boundaries is entirely composed of equilibria, and stability is determined by the direction of the vector field near it.  Physically, this degenerate scenario comes about when there is an absence of volcanism in the ice-covered state or a perfect balance between weathering and volcanism in the ice-free state. We do not discuss these cases further.\end{remark}

\begin{figure}[ht]
\centering
\includegraphics[width=0.47\textwidth]{{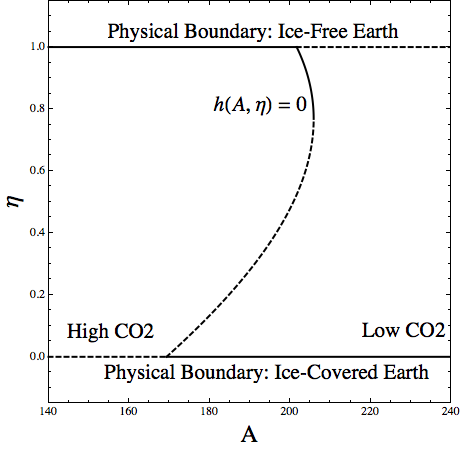}} 
\caption{The physical region of the phase space and possible fixed points of the system given by the $\eta$-nullcline, $h=0$.  The location of the equilibrium is determined by the critical effective area of exposed land $0<\eta_c<1$. Solid black portions of the curve represent stable equilibria while the dashed lines denote unstable equilibria.  Solid black portions of the boundary are attractive sliding regions and dashed boundaries are crossing regions.}\label{critman}
\end{figure} 

\subsubsection{Sliding and escape from Snowball Earth}

Recall that $\eta_c$ was defined to be a ratio of volcanism to weathering.  We now prove that the system always recovers from the ice-covered regime provided this parameter is positive, i.e. there is some contribution from volcanic outgassing to atmospheric greenhouse gas content. {The following result can be easily achieved following the proof of Theorem \ref{thm:projection}:}

\begin{corollary}\label{exit} Suppose $\eta_c>0$.  Then any orbit of the system \eqref{AetaPW} that enters the line $\eta=0$ must exit in finite time.\end{corollary}

\begin{proof}
{Since $\frac{dA}{dt}\Big|_{\eta=0}<0$, any orbit that reaches the line $\eta=0$ must eventually reach the unique point $(A_*,0)$ such that $h(A_*,0)=0$. Since $h$ is negative to the right of this point and positive to the left, there is a single orbit that reaches the tangency point from the interval $0<\eta<1$, passes through it from right to left, and then reenters the open interval. Due to the forward-time uniqueness guaranteed by Theorem \ref{thm:projection}, the sliding solution must follow this orbit and reenter the physical region.}
\end{proof}
 
\begin{remark} Similarly, any orbit that enters an ice-free state must exit in finite time.  The required condition is that there is \textit{not} a perfect balance between volcanism and weathering, i.e. $\eta_c\neq 1$.\end{remark}

\subsubsection{Numerical simulations\label{po}}

As mentioned above, the Filippov framework  allows for more robust numerical simulation of the model.  Let $(A_c,\eta_c)$ be the location of the single fixed point.  We find that small ice cap states corresponding to fixed points with $\eta_f<\eta_c<1$ are possible attractors of the system, as are periodic orbits born from the Hopf bifurcation at $\eta_c=\eta_f$.  

The fold of the $\eta$-nullcline is a \textit{canard point}, and the mechanism that produces the large amplitude periodic orbit in Figure \ref{dynamics} is a \textit{canard explosion} \cite{MR643399,MR653888,MR653889,MR653890, dumortier1996canard, krupa2001extending}. While mathematically interesting, canard orbits are unlikely observables for planar systems because they occur within a parameter range \added{that is exponentially small in $\delta$.}  However, they can be robust in higher dimensions which makes more complex oscillatory behavior possible, see e.g.  \cite{benoit1983systemes,desroches2012mixed,szmolyan2001canards,wechselberger2005existence}.  With this in mind, we remark that an interesting avenue for future research involves studying an extended version of the system \eqref{ig-pws}-\eqref{ig-pws1} with an additional variable ($w$ in \cite{mcgwid2014simplification}) and foregoing their invariant manifold reduction.

\added{In the simulations, trajectories that leave the upper boundary and re-enter the physical region are rapidly attracted to a slow manifold that is within $\mathcal{O}(\delta)$ of the $\eta$-nullcline (the so-called \textit{critical manifold}) as guaranteed by Fenichel theory \cite{fenichel1979geometric, jones1995geometric}.  This is why the trajectories in the figures appear to have sharp corners at the exit from Snowball Earth, when in fact they are differentiable at the exit point and tangent to the boundary.}  

 \begin{figure}[ht]
  \centering
   \subfloat[Small ice cap equilibrium]  {\includegraphics[width=0.4\textwidth]{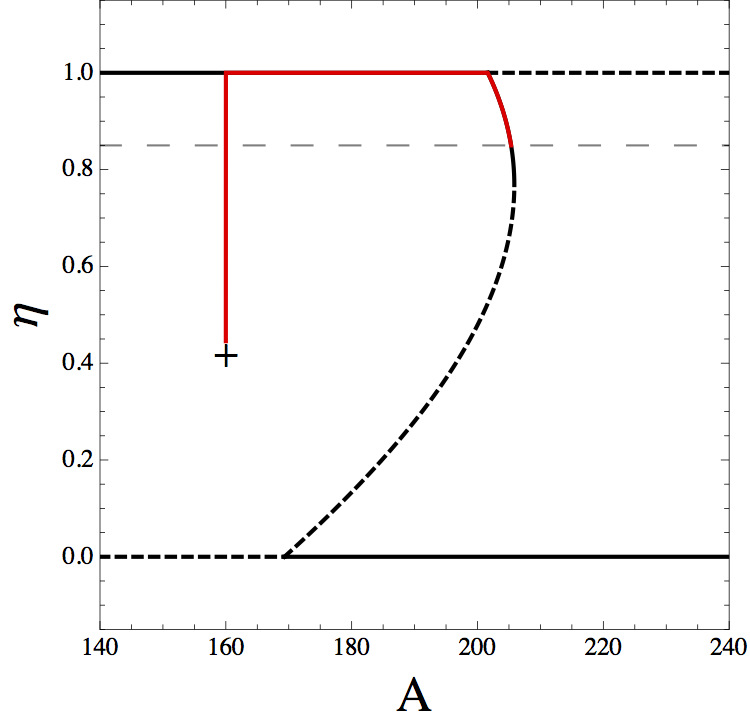}}  \quad \quad
   \subfloat [Periodic orbit]  {\includegraphics[width=0.4\textwidth]{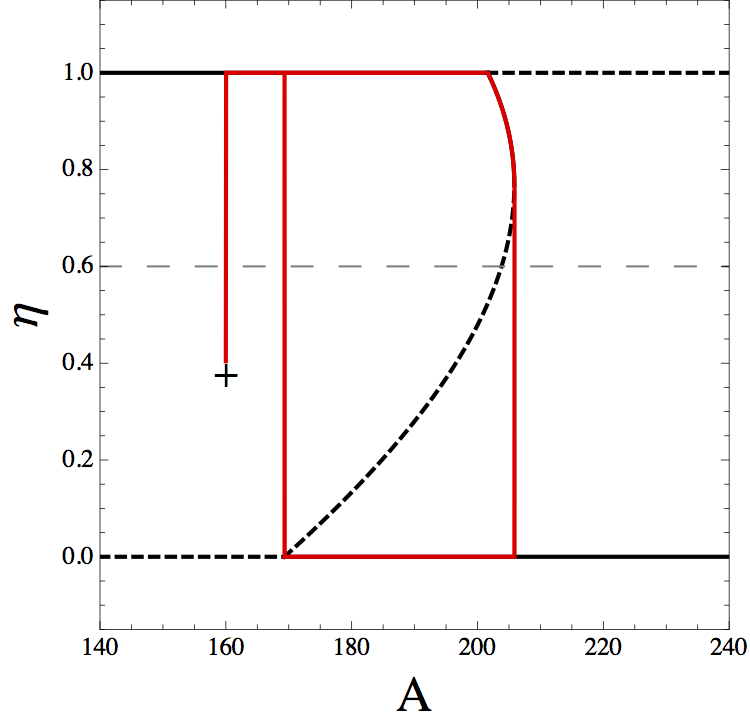}} \quad
  \caption{Attractors of the system when (a) $\eta_c=0.85$ and (b) $\eta_c=0.6$. The $+$ symbol marks the initial condition and the horizontal long-dashed line is the $A$-nullcline. In (a), the orbit reaches the ice-free state and slides until it reaches the intersection of the folded curve $h(A,\eta)=0$ with this boundary.  It then enters the physical region and approaches the small ice cap equilibrium.  In (b), the fixed point is unstable and the orbit oscillates between the ice-free and ice-covered boundaries.  \added{Parameters are as in Table \ref{tab:ParValues} and $\delta=0.01$.} Simulations were performed using Mathematica 9.}  
\label{dynamics}
 \end{figure}

\subsection{Application to the Jormungand Model\label{sect:jorm}}

The previous occurrence of a complete snowball event is still a highly contested topic in geology.  In the model \eqref{ig-pws}-\eqref{ig-pws1}, the only possible stable fixed point is a small ice cap. In this section, we present a variant of the system, as introduced by Abbot, Voigt, and Koll \cite{abbot2011}, and employ the {nonsmooth} framework as above.  By modifying the albedo function $\alpha(\eta,y)$ so that it differentiates between the albedo of bare ice and snow-covered ice, the new system has an additional stable state that corresponds to a large ice cap. In \cite{abbot2011}, this was called the \textit{Jormungand} state because it allows for a snake-like band of open ocean at the equator.  Moreover, there are again attracting periodic orbits.  These can be seen in Figure \ref{jorm}.

In this version of the model, the albedo is defined as follows:

$$\alpha_J(\eta,y)=\left\{\begin{array}{ll}\alpha_w,& y<\eta\\
\frac{1}{2}(\alpha_w+\alpha_2(\eta)), & y=\eta\\ \alpha_2(y), & y>\eta, \end{array}\right.\eqno(14)
$$

where $\alpha_2(y)=\frac{1}{2}(\alpha_s+\alpha_i)+\frac{1}{2}(\alpha_s-\alpha_i)\tanh M(y-0.35) $.

Here $\alpha_w$ is the albedo of open water, $\alpha_i$ is the albedo of  {\em bare} sea ice, and $\alpha_s$ is the albedo of {\em snow-covered} ice. The model assumes sea ice aquires a snow cover only for latitudes above $y=0.35$.  We modify $h(A,\eta)$ by replacing $\alpha$ with $\alpha_J$:

\begin{equation}\label{hj}
h_J(A,\eta)=\frac{Q}{B+C} \left(s(\eta)(1-\alpha_J(\eta,\eta))+\frac{C}{B}(1-\overline{\alpha_J}(\eta))\right)-\frac{A}{B}-T_c 
\end{equation}
where we note that although $\alpha_J(\eta,y)$ has a discontinuity across $y=\eta$, both $\alpha_J(\eta,\eta)$ and $\overline{\alpha_J}(\eta)=\int_0^1 \alpha(\eta,y)s(y)dy$ are smooth functions of $\eta$.

For this system, we can immediately apply Theorem \ref{thm:projection} and {Corollary} \ref{exit} to obtain the following result.

\begin{corollary} The physical region is forward invariant with respect to the system \eqref{AetaPW} where $h$ is replaced by $h_J$. Solutions exist and are unique in forward time, and any trajectory that enters an ice-covered state must exit in finite time provided $\eta_c\neq 0$.\end{corollary}

 \begin{figure}[ht]
  \centering
   \subfloat[]  {\includegraphics[width=0.4\textwidth]{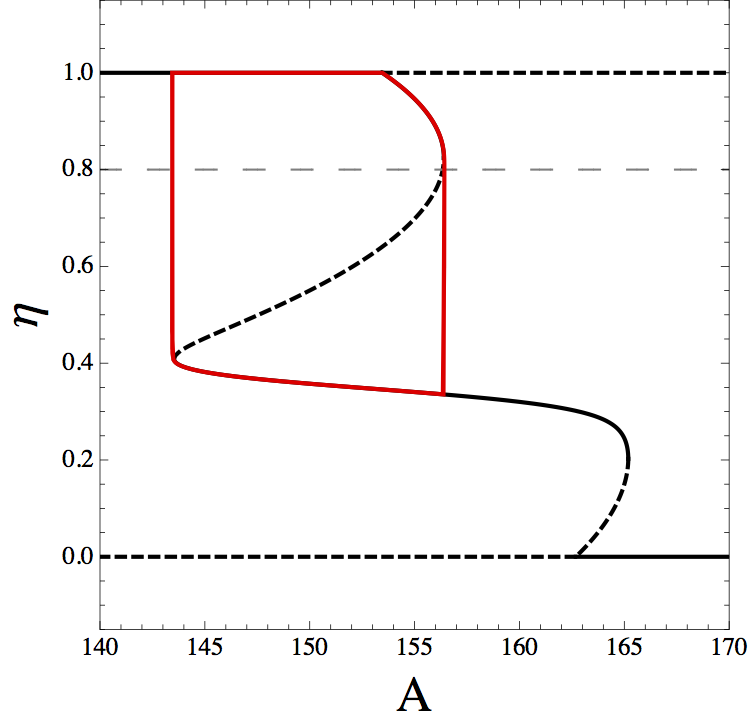}}  \quad \quad
   \subfloat []  {\includegraphics[width=0.4\textwidth]{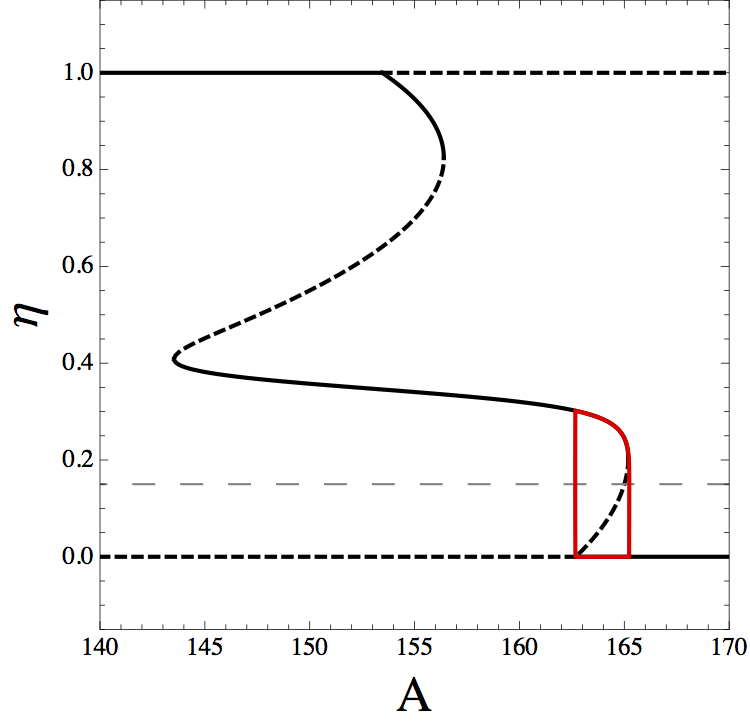}} \quad
  \caption{Periodic orbits of the Jormungand system when (a) $\eta_c=0.8$ and (b) $\eta_c=0.15$. The folded curve is the $\eta$-nullcline $h_J(A,\eta)=0$ and dashing is as in Figures \ref{critman} and \ref{dynamics}.  \added{Parameters are as in Table \ref{tab:ParJorm} and $\delta=0.01$.}}  
\label{jorm}
 \end{figure}

\begin{table}[htb]  
\centering
\begin{tabular}{|ccc||c|}  
\hline
\textbf{Parameters} & Value & Units \\ \hline \hline
&& \\
$T_c$ & 0 & ${}^\circ\text{C}$  \\
$M$ & 25 & dimensionless  \\
$\alpha_w$ & $0.35$ & dimensionless   \\
$\alpha_i$ & $0.45$ & dimensionless   \\
$\alpha_s$ & $0.8$ & dimensionless   \\
&& \\
\hline
\end{tabular}  
\caption{Parameter values as in Table \ref{tab:ParValues} unless specified above.  Additional values taken from \cite{abbot2011}.}   \label{tab:ParJorm}
\end{table}

\section{Discussion\label{sect:discussion}}

In this article we have extended a class of energy balance models to include a greenhouse gas component.  {The resulting system was nonsmooth due to physical constraints on the extent of Earth's glaciation and did not immediately fit into an existing analytical framework.  However, we defined two frameworks that produced the (same) expected motion, in agreement with physical arguments from the climate literature.  Both approaches prevented the dynamics from exiting the physically relevant region; the first was based on a rule that projected the vector field onto the boundary when it was reached and the second defined a ``confinement vector field'' that pointed into the region.  In both cases, we showed the existence and uniqueness of forward-time solutions.  We further proved that the system always escapes from the ice-covered scenario, in agreement with Kirschvink's hypothesis about carbon dioxide accumulation in a snowball scenario due to a shut down of chemical weathering processes \cite{kirschvink1992protero}.  Using our model, we found that small ice-cap and large amplitude ice-cover oscillations were possible attractors of the system. We then applied our results to the case of a ``Jormungand'' world and showed that it is possible to obtain further oscillatory dynamics between no ice cover, large ice-cover, and full ice-cover states. }

General mathematical questions brought about by this work have to do with building a general framework for such models.  A first step toward dealing with physical boundaries might be a general theory for semiflows generated by smooth vector fields on manifolds with boundary.  In addition, the stable portions of the physical boundaries in our model should be thought of as equilibria of the fast subsystem, i.e. as part of the \textit{critical manifold} from geometric singular perturbation theory \cite{jones1995geometric}.  In fact, they are much more \textit{normally hyperbolic} than the curve $h=0$; they are reachable in finite time!  However, the current theory cannot be directly applied to the full discontinuous system and we remark that developing an analogous Fenichel theory \cite{fenichel1979geometric} for singularly perturbed discontinuous systems with sliding is an interesting avenue for future research.

Another interesting future direction is to study the explicit effect of the temperature (here we have considered it only through the equation for $\eta$ based on the reduction by McGehee and Widiasih \cite{mcgwid2014simplification}) on the dynamics of the system.  In a nonsmooth system, the addition of a stable dimension may destroy an existing stable periodic orbit \cite{sieber2010small}.  Moreover, an additional dimension could result in more interesting glacial dynamics, such as mixed mode oscillations.
 
Finally, it is natural to ask how the shape of the $\eta$-nullcline changes with physical parameters and how this affects system dynamics, e.g. could the lower fold in Figure \eqref{jorm} move to the left of the upper fold?  We refer the reader to \cite{widwal2014dynamics} for a number of examples.  \\

\section*{Acknowledgements:} This research was supported in part by the Mathematics and Climate Research Network and NSF grants DMS-0940366, DMS-0940363.  AB was also supported in part by the Institute for Mathematics and its Applications with funds provided by the National Science Foundation.  We thank the members of the MCRN Paleoclimate and Nonsmooth Systems seminar groups for many useful discussions, especially Mary Lou Zeeman and Emma Cutler.  We are also grateful to Mike Jeffrey, Rachel Kuske, Andrew Roberts, and the anonymous referee for their insights and suggestions.

\bibliographystyle{plain}
      \bibliography{ghg.bib}

\begin{thebibliography}{10}

\bibitem{abbot2011}
D.~Abbot, A.~Voigt, and D.~Koll.
\newblock The {J}ormungand global climate state and implications for
  {N}eoproterozoic glaciations.
\newblock {\em J Geophys. Res., 116}, 2011.

\bibitem{MR653888}
{\'E}.~Beno{\^\i}t.
\newblock Chasse au canard. {II}. {T}unnels---entonnoirs---peignes.
\newblock {\em Collect. Math.}, 32(2):77--97, 1981.

\bibitem{benoit1983systemes}
{\'E}~Beno{\^\i}t.
\newblock Syst\`emes lents-rapides dans $\mathbb{R}^3$\ et leurs canards.
\newblock In {\em Third {S}chnepfenried geometry conference, {V}ol. 2
  ({S}chnepfenried, 1982)}, volume 109 of {\em Ast\'erisque}, pages 159--191.
  Soc. Math. France, Paris, 1983.

\bibitem{MR653890}
{\'E}.~Beno{\^\i}t and J.-L. Callot.
\newblock Chasse au canard. {IV}. {A}nnexe num\'erique.
\newblock {\em Collect. Math.}, 32(2):115--119, 1981.

\bibitem{budyko2010effect}
M.~I. Budyko.
\newblock The effect of solar radiation variations on the climate of the earth.
\newblock {\em Tellus}, 21(5):611--619, 1969.

\bibitem{cahalan1979stability}
R.~F. Cahalan and G.~R. North.
\newblock A stability theorem for energy-balance climate models.
\newblock {\em Journal of the Atmospheric Sciences}, 36(7):1178--1188, 1979.

\bibitem{caldeira1992susceptibility}
K.~Caldeira and J.~F. Kasting.
\newblock Susceptibility of the early earth to irreversible glaciation caused
  by carbon dioxide clouds.
\newblock {\em Nature}, 359(6392):226--228, 1992.

\bibitem{MR653889}
J.-L. Callot.
\newblock Chasse au canard. {III}. {L}es canards ont la vie br\`eve.
\newblock {\em Collect. Math.}, 32(2):99--114, 1981.

\bibitem{caratheodory2013vorlesungen}
C.~Carath{\'e}odory.
\newblock {\em Vorlesungen {\"u}ber reelle Funktionen}.
\newblock Leipzig, 1927.

\bibitem{desroches2012mixed}
M.~Desroches, J.~Guckenheimer, B.~Krauskopf, C.~Kuehn, H.~Osinga, and
  M.~Wechselberger.
\newblock Mixed-mode oscillations with multiple time scales.
\newblock {\em SIAM Review}, 54(2):211--288, 2012.

\bibitem{MR643399}
F.~Diener and M.~Diener.
\newblock Chasse au canard. {I}. {L}es canards.
\newblock {\em Collect. Math.}, 32(1):37--74, 1981.

\bibitem{dumortier1996canard}
F.~Dumortier and R.~H. Roussarie.
\newblock {\em Canard cycles and center manifolds}, volume 577.
\newblock American Mathematical Soc., 1996.

\bibitem{edmond1996fluvial}
J.~M. Edmond, M.~R. Palmer, E.~T. Brown, and Y.~Huh.
\newblock Fluvial geochemistry of the eastern slope of the northeastern andes
  and its foredeep in the drainage of the orinoco in colombia and venezuela.
\newblock {\em Geochimica et cosmochimica acta}, 60(16):2949--2974, 1996.

\bibitem{fenichel1979geometric}
N.~Fenichel.
\newblock Geometric singular perturbation theory for ordinary differential
  equations.
\newblock {\em Journal of Differential Equations}, 31(1):53--98, 1979.

\bibitem{filippov1964}
A.~F. Filippov.
\newblock Differential equations with discontinuous right-hand side.
\newblock {\em American Mathematical Society Translations}, 2:199--231, 1964.

\bibitem{filippov1988}
A.~F. Filippov.
\newblock {\em Differential Equations with Discontinuous Righthand Sides}.
\newblock Kluwer Academic Publ. Dortrecht, 1988.

\bibitem{graves1993new}
C.~E. Graves, W.-H. Lee, and G.~R. North.
\newblock New parameterizations and sensitivities for simple climate models.
\newblock {\em Journal of geophysical research}, 98(D3):5025--5036, 1993.

\bibitem{hoffmanschrag2002terranova}
P.~Hoffman and D.~Schrag.
\newblock The snowball earth hypothesis: testing the limits of global change.
\newblock {\em Terra Nova, 14}, page 129â€“155, 2002.

\bibitem{hogg}
A.~M. Hogg.
\newblock Glacial cycles and carbon dioxide: A conceptual model.
\newblock {\em Geophysical Research Letters}, 35(1):L01701, 2008.

\bibitem{jeffrey2014hidden}
M.~R. Jeffrey.
\newblock Hidden dynamics in models of discontinuity and switching.
\newblock {\em Physica D: Nonlinear Phenomena}, 273:34--45, 2014.

\bibitem{jones1995geometric}
C.~K. R.~T. Jones.
\newblock Geometric singular perturbation theory.
\newblock {\em Dynamical systems}, pages 44--118, 1995.

\bibitem{kirschvink1992protero}
J.~Kirschvink.
\newblock Late proterozoic low-latitude global glaciation: the snowball earth.
\newblock {\em The Proterozoic Biosphere: A Multidisciplinary Study}, pages
  51--52, 1992.

\bibitem{krupa2001extending}
M.~Krupa and P.~Szmolyan.
\newblock Extending geometric singular perturbation theory to nonhyperbolic
  points---fold and canard points in two dimensions.
\newblock {\em SIAM journal on mathematical analysis}, 33(2):286--314, 2001.

\bibitem{kump2000chemical}
L.~R. Kump, S.~L. Brantley, and M.~A. Arthur.
\newblock Chemical weathering, atmospheric co2, and climate.
\newblock {\em Annual Review of Earth and Planetary Sciences}, 28(1):611--667,
  2000.

\bibitem{mcglehman2012}
R.~McGehee and C.~Lehman.
\newblock A paleoclimate model of ice-albedo feedback forced by variations in
  earth's orbit.
\newblock {\em SIAM Journal on Applied Dynamical Systems}, 11(2):684--707,
  2012.

\bibitem{mcgwid2014simplification}
R.~McGehee and E.~Widiasih.
\newblock A quadratic approximation to {B}udyko's ice albedo feedback model
  with ice line dynamics.
\newblock {\em SIADS, Vol. 13, Issue 1, DOI: 10.1137/120871286}, 2014.

\bibitem{north1975theory}
G.~R. North.
\newblock Theory of energy-balance climate models.
\newblock {\em J. Atmos. Sci}, 32(11):2033--2043, 1975.

\bibitem{pierrehumbert2011climate}
R.~T. Pierrehumbert, D.~S. Abbot, A.~Voigt, and D.~Koll.
\newblock Climate of the {N}eoproterozoic.
\newblock {\em Annual Review of Earth and Planetary Sciences}, 39:417--460,
  2011.

\bibitem{pollard2005snowball}
D.~Pollard and J.~F. Kasting.
\newblock Snowball {E}arth: A thin-ice solution with flowing sea glaciers.
\newblock {\em Journal of Geophysical Research: Oceans (1978--2012)}, 110(C7),
  2005.

\bibitem{sellersglobal}
W.~D. Sellers.
\newblock A global climatic model based on the energy balance of the
  earth-atmosphere system.
\newblock {\em Journal of Applied Meteorology}, 8(3):392--400, 1969.

\bibitem{sieber2010small}
J.~Sieber and P.~Kowalczyk.
\newblock Small-scale instabilities in dynamical systems with sliding.
\newblock {\em Physica D: Nonlinear Phenomena}, 239(1):44--57, 2010.

\bibitem{szmolyan2001canards}
P.~Szmolyan and M.~Wechselberger.
\newblock Canards in $\mathbb{R}^3$.
\newblock {\em Journal of Differential Equations}, 177(2):419--453, 2001.

\bibitem{tung2007topics}
K.~K. Tung.
\newblock {\em Topics in mathematical modelling}.
\newblock Princeton University Press, 2007.

\bibitem{widwal2014dynamics}
J.~Walsh and E.~Widiasih.
\newblock A dynamics approach to a low-order climate model.
\newblock {\em Discrete \& Continuous Dynamical Systems-Series B}, 19(1), 2014.

\bibitem{wechselberger2005existence}
M.~Wechselberger.
\newblock Existence and bifurcation of canards in $\mathbb{R}^3$ in the case of
  a folded node.
\newblock {\em SIAM Journal on Applied Dynamical Systems}, 4(1):101--139, 2005.

\bibitem{widiasih2013dynamics}
E.~R. Widiasih.
\newblock Dynamics of the {B}udyko energy balance model.
\newblock {\em SIAM Journal on Applied Dynamical Systems}, 12(4):2068--2092,
  2013.

\end{thebibliography}
\end{document}